\documentclass[12pt,a4paper]{amsart}
\usepackage{amsfonts}
\usepackage{amsthm}
\usepackage{amsmath}
\usepackage{amscd}
\usepackage[latin2]{inputenc}
\usepackage{t1enc}
\usepackage[mathscr]{eucal}
\usepackage{indentfirst}
\usepackage{graphicx}
\usepackage{graphics}
\usepackage{subcaption}

\numberwithin{equation}{section}

\theoremstyle{plain}
\newtheorem{Th}{Theorem}[section]

\newtheorem{Cor}[Th]{Corollary}

\newtheorem{Pro}[Th]{Proposition}

 \theoremstyle{definition}

\newtheorem{?}[Th]{Problem}

\DeclareMathSizes{5}{5}{5}{5}

\usepackage{float}
\usepackage{tikz}
\usetikzlibrary{arrows}
\tikzstyle{n}=[circle,draw,thick,scale=0.5]
\tikzstyle{nt}=[circle,draw,thick,fill,scale=0.5]
\tikzstyle{nc}=[circle,draw,thick,fill]

\begin{document}

\title{One more remark on the adjoint polynomial}
\thanks{The author was partially supported by the MTA R\'enyi Institute Lend\"ulet Limits of Structures Research Group.}

\author{Ferenc Bencs} 
 \address{Central European University, Department of Mathematics
 \\ H-1051 Budapest
 \\ Zrinyi u. 14. \\ Hungary \& Alfr\'ed R\'enyi Institute of Mathematics\\ H-1053 Budapest\\ Re\'altanoda u. 13-15.} 
 \email{ferenc.bencs@gmail.com}



\begin{abstract}
 The adjoint polynomial of $G$ is \[h(G,x)=\sum_{k=1}^n(-1)^{n-k}a_k(G)x^k,\] where $a_k(G)$ denotes the number of ways one can cover all vertices of the graph $G$ by exactly $k$ disjoint cliques of $G$. In this paper we show the the adjoint polynomial of a graph $G$ is a simple transformation of the independence polynomial of another graph $\widehat{G}$. This enables us to use the rich theory of independence polynomials to study the adjoint polynomials. In particular we a give new proofs of several theorems of R. Liu and P. Csikv\'{a}ri.
\end{abstract}

\maketitle

\section{Introduction} 
Let $a_k(G)$ denote the number of ways one can cover all vertices of the graph $G$ by exactly $k$ disjoint cliques of $G$. From the definition it is clear that $a_n(G)=1$ and $a_{n-1}(G)=e(G)$, the number of edges, where the number of vertices of $G$ is $n$. Then the adjoint polynomial of $G$ is
\[
  h(G,x)=\sum_{k=1}^n(-1)^{n-k}a_k(G)x^k.
\]
The adjoint polynomial  was introduced by R. Liu \cite{Liu1997} and it is studied in a series of papers 
(\cite{Brenti1992, Brenti1994, Zhao2004, Zhao2005, Zhao2009}).
Let us remark that we are using the definition for the adjoint polynomial from \cite{Csikvari2012a}, but usually it is defined without alternating signs.  

The adjoint polynomial shows a strong connection with the chromatic polynomial \cite{Read1968}. More precisely the chromatic polynomial of the complement graph $\bar{G}$ of $G$ is 
\[
  ch(\bar{G},x)=\sum_{k=1}^na_k(G)x(x-1)\dots(x-k+1).
\]

The adjoint polynomial shows  certain  nice analytic properties. For instance it has a real zero whose modulus is the largest among all zeros. H. Zhao showed (\cite{Zhao2005}) that the adjoint polynomial always has a real zero, furthermore P. Csikv\'ari proved (\cite{Csikvari2012a}) that the largest real zero has the largest modulus among all zeros. He also showed that the absolute value of the largest real zero is at most $4(\Delta -1)$, where $\Delta$ is the largest degree of the graph $G$. 

Similar results hold for the independence polynomial. Recall that the independence polynomial of a graph $H$ is \[I(H,x)=\sum_{k\ge 0}(-1)^ki_k(H)x^k,\] where $i_k(H)$ denotes the number of independent sets  of size $k$ in $H$ (note that $i_0(H)=1$). It suggests that there might be a connection between the independence polynomials and the adjoint polynomials. 

In this paper we will show that there is indeed such a connection between the two graph polynomials. We will prove the following theorem:
\begin{Th}\label{th:main}
 For any graph $G$ there exists a graph $\widehat{G}$, such that 
 \[
  h(G,x)=x^{n}I(\widehat{G},1/x).
 \]
\end{Th}
This correspondence will enable us to  use the rich theory of independence polynomials to study the adjoint polynomials. In particular, we give new proofs of the aforementioned results R. Liu and P. Csikv\'ari. For details see Section ~\ref{sec:m}.

In Section ~\ref{sec:cons} we will give the construction of $\widehat{G}$ and prove Theorem~\ref{th:main}, moreover, we will show that $\widehat{G}$ can be taken as a spanning subgraph of the line graph of $G$. Recall that the line graph $L(G)$ for a graph $G$ is a graph on the edge set of $G$, and there is an edge between two vertices of the line graph if they share a common vertex. This enables us to establish a connection with the matching polynomial of the graph $G$. For definition of the matching polynomial and applications of this connection see Section ~\ref{sec:m}.

\subsection{Notation} Denote the vertex set and edge set of a graph $G$ by $V(G)$ and $E(G)$, respectively.
Let $N_G(u)$ denote the set of neighbors of the vertex $u$ and $d(u)$ the degree of the vertex $u$.
Let $N_G[u]=N_G(u)\cup\{u\}$ denote the closed neighborhood of the vertex $u$. If it is clear from the context, then we will write $N(u)$ and $N[u]$ instead of $N_G(u)$ and $N_G[u]$.
Let $G-v$ denote the graph obtained from $G$ by deleting the vertex $v$.
If $S\subseteq V(G)$, then $G[S]$ denotes the induced subgraph of $G$ on the vertex set $S$, and $G-S$ denotes $G[V(G)-S]$.

\section{The construction}\label{sec:cons}
In this section we will construct $\widehat{G}$ and prove Theorem~\ref{th:main}.

\medskip
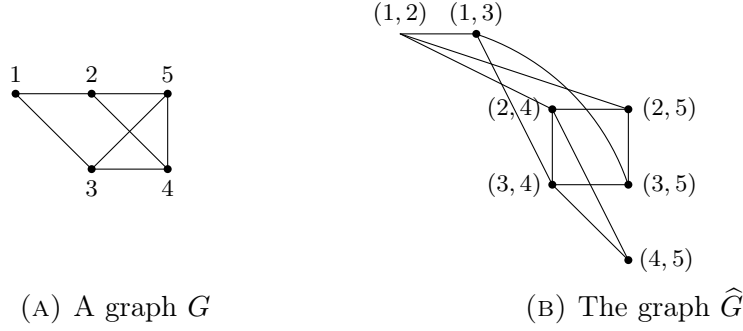
\begin{figure}[h]
      \begin{subfigure}[b]{0.5\linewidth}
	\centering
	\begin{tikzpicture}[line cap=round,line join=round,>=triangle 45,x=1.0cm,y=1.0cm]
	\clip(0.68,0.58) rectangle (3.8,2.68);
	\draw (1,2)-- (2,2);
	\draw (2,2)-- (3,2);
	\draw (1,2)-- (2,1);
	\draw (2,1)-- (3,2);
	\draw (3,2)-- (3,1);
	\draw (2,1)-- (3,1);
	\draw (3,1)-- (2,2);
	\begin{scriptsize}
	\fill  (1,2) circle (1.5pt);
	\draw (1,2.26) node {$1$};
	\fill  (2,2) circle (1.5pt);
	\draw (2,2.26) node {$2$};
	\fill  (2,1) circle (1.5pt);
	\draw (2,0.75) node {$3$};
	\fill  (3,1) circle (1.5pt);
	\draw (3,0.75) node {$4$};
	\fill  (3,2) circle (1.5pt);
	\draw (3,2.26) node {$5$};
	\end{scriptsize}
	\end{tikzpicture}
	\vspace{1cm}
	\caption{A graph $G$}
      \end{subfigure}%
      \begin{subfigure}[b]{0.5\linewidth}
	\begin{tikzpicture}[line cap=round,line join=round,>=triangle 45,x=1.0cm,y=1.0cm]
	  \clip(0.66,0.84) rectangle (5.18,4.6);
	  \draw (1,4)-- (2,4);
	  \draw (1,4)-- (3,3);
	  \draw (1,4)-- (4,3);
	  \draw (2,4)-- (3,2);
	  \draw (3,3)-- (3,2);
	  \draw (3,2)-- (4,1);
	  \draw (3,3)-- (4,1);
	  \draw (3,3)-- (4,3);
	  \draw (3,2)-- (4,2);
	  \draw (4,2)-- (4,3);
	  \draw [shift={(1,1)}] plot[domain=0.32:1.25,variable=\t]({1*3.16*cos(\t r)+0*3.16*sin(\t r)},{0*3.16*cos(\t r)+1*3.16*sin(\t r)});
	  \begin{scriptsize}
	  \draw (1,4.26) node {$(1,2)$};
	  \fill  (2,4) circle (1.5pt);
	  \draw (2,4.26) node {$(1,3)$};
	  \fill  (3,3) circle (1.5pt);
	  \draw (2.5,3) node {$(2,4)$};
	  \fill  (3,2) circle (1.5pt);
	  \draw (2.5,2) node {$(3,4)$};
	  \fill  (4,1) circle (1.5pt);
	  \draw (4.5,1) node {$(4,5)$};
	  \fill  (4,2) circle (1.5pt);
	  \draw (4.54,2) node {$(3,5)$};
	  \fill  (4,3) circle (1.5pt);
	  \draw (4.54,3) node {$(2,5)$};
	  \end{scriptsize}
	  \end{tikzpicture}
	  \caption{The graph $\widehat{G}$}
      \end{subfigure}
\caption{An example for the construction}
\label{fig:cons}
\end{figure}

Let $\mathcal{K}(G)$ denote the set of clique covers of $G$, that is 
\begin{align*}
  \{\{S_1,\dots,S_k\}\subseteq \mathcal{P}(V(G))~|&~\cup_{i=1}^kS_i=V(G),~\textrm{if $1\le i\neq j\le n$, then }\\& S_i\cap S_j=\emptyset, ~G[S_i] \textrm{ is a complete graph}\},
\end{align*}
then one can write the adjoint polynomial of $G$ in the following way:
\[
  h(G,x)=\sum_{Q\in \mathcal{K}(G)}(-1)^{n-|Q|}x^{|Q|}.
\]
We will also use the notation  
\[
  h^*(G,x)=x^nh\left(G,1/x\right)=\sum_{k=0}^{n-1}(-1)^ka_{n-k}(G)x^k.
\]
Let us choose an arbitrary ordering on the vertices of $G$, that is $V(G)=\{u_1,\dots,u_n\}$. Then we construct a $\widehat{G}$ graph as follows. Let $V(\widehat{G})=\{(u_i,u_j)\in E(G)~|~1\le i<j\le n\}$, and let $(u_i,u_j)\neq (u_k,u_l)\in V(\widehat{G})$ be two vertices. We may assume that $j\le l$, then the two vertices are connected, if and only if ($i=k$) or ($j=k$) or ($j=l$ and $(u_i,u_k)\notin E(G)$). Clearly $\widehat{G}$ is a subgraph of the line graph of $G$. In the next theorem we show that  the independence polynomial of $\widehat{G}$ actually equals to $h^*(G,x)$. For example see Figure ~\ref{fig:cons}.

\begin{Pro}\label{Th:main}
 Let $G$ be a graph and let us choose an ordering of the vertices. Then the constructed $\widehat{G}$ graph satisfies that
 \[
  h^*(G,x)=I(\widehat{G},x).
 \]
 Moreover if $e=(u_{n-1},u_n)\in E(G)$, then $\widehat{G-e}\subseteq \widehat{G}$.
\end{Pro}

\begin{proof}
  The second statement is clear from the construction above. In order to prove the first statement we will show that there is a bijection between independent sets of $\widehat{G}$ and $\mathcal{K}=\mathcal{K}(G)$. More precisely, if we let $\mathcal{I}=\mathcal{I}(\widehat{G})$ denote the set of independent sets of $\widehat{G}$, then there exists a  bijection $\phi:\mathcal{K}\to\mathcal{I}$ such that for any $Q\in\mathcal{K}$ we have $|\phi(Q)|=n-|Q|$.
  
  Let $Q=\{S_1,\dots,S_k\}\in\mathcal{K}$. For $1\le i\le k$ let $f(i)$ denote the maximal index in $S_i$, that is $f(i)=\max\{1\le j\le n ~|~ u_j\in S_i\}$. Then $\phi(Q)$ will be the union of clique edges having one endpoint as a maximally indexed vertex, that is
  \[
    \phi(Q)=\bigcup_{1\le i\le k}\{(u_j,u_{f(i)})~|~u_j\in S_i, ~j\neq f(i)\}\subseteq V(\widehat{G}).
  \]
  
  First we show that $\phi(Q)$ is an independent set in $\widehat{G}$. Let us define the sets $F_i=\{(u_j,u_{f(i)})~|~u_j\in S_i, ~j\neq f(i)\}$ for $1\le i\le k$. If $1\le i<i'\le k$, then there is no edge between $F_i$ and $F_{i'}$ in $\widehat{G}$, since $Q$ is a partition of $V(G)$. Also the set $F_i$ is independent for  $1\le i\le k$, since if $(u_j,u_{f(i)})\neq(u_{j'},u_{f(i)})\in F_i\subseteq E(G[S_i])$, then $(u_j,u_{j'})\in E(G[S_i])$, because $S_i$ is a clique. 
  
  Furthermore we see that $|F_i|=|S_i|-1$, so 
  \[
    |\phi(Q)|=\sum_{1\le i\le k}|F_i|=\sum_{1\le i\le k}(|S_i|-1)=n-|Q|
  \]
  
  For the surjectivity of $\phi$ let $I\in \mathcal{I}$ be fixed, and let $K_i=\{(u_j,u_i)\in I~|~j<i\}$ for $1\le i\le n$. Then 
  \[
    Q'=\left\{ \{u_i\}\cup\{u_j~|~(u_j,u_i)\in K_i\}~|~K_i\neq\emptyset \right\}
  \]
  is a set of pairwise disjoint subsets of $V(G)$, where each subset induces a clique in $G$. Therefore
  \[
   Q=Q'\cup \left\{\{u_i\}~|~u_i\notin \cup Q'\right\}
  \]
  is a partition of $V(G)$ where each part induces a clique in $G$. Moreover we have that
  \[
    \phi(Q)=\bigcup_{K_i\neq \emptyset}\{(u_j,u_i)\in I~|~j<i\}=I.
  \] 
  So $\phi$ is a bijection.

  Now the statement follows as:
  \begin{eqnarray*}
   h^*(G,x)=\sum_{Q\in \mathcal{K}}(-x)^{n-|Q|}=\sum_{Q\in\mathcal{K}}(-x)^{|\phi(Q)|}=\sum_{I\in\mathcal{I}}(-x)^{|I|}=I(\widehat{G},x).
  \end{eqnarray*}  
  
\end{proof}

\section{About ``the largest'' root of the adjoint polynomial}\label{sec:m}

In this section we will give a various applications of Theorem~\ref{th:main} and the construction. 
First we collect some results on independence polynomials of graphs. 


\begin{Th}\cite{Fisher1990,Goldwurm2000}\label{th:existbeta}
 Let $G$ be a connected graph. Then  $I(G,x)$ has a zero in the interval $(0,1]$, and  let $\beta(G)$ be the smallest  among them. Then $\beta(G)$ is a simple zero of $I(G,x)$, and if $\xi\neq\beta(G)$ is a zero of $I(G,x)$, then $\beta(G)<|\xi|$.
\end{Th}

\begin{Th}\cite{Csikvari2012}\label{th:betamonoton}
 Let $G$ be a connected graph and $H$ be a proper subgraph of  $G$. 
Then $\beta(H)>\beta(G)$.
\end{Th}

\begin{Th}\cite{Csikvari2012}\label{th:betamoment}
 Let $G$ be a connected graph and $H$ be an induced subgraph, then in the following series
  \[
    \frac{I(H,x)}{I(G,x)}=\sum_{k\ge 0} r_{k}(G,H) x^k,
  \]
  for each $k\ge 0$ the coefficients $r_k(G,H)$ are positive integers.  
\end{Th}

We will also need some results on a modified version of the matching polynomial. Let 
\[
  M(G,x)=\sum_{k\ge 0}(-1)^km_k(G)x^{n-k},
\] where $m_k(G)$ denotes the number of matchings with $k$ edges in  $G$. Note that
\[
  M(G,x)=x^nI(L(G),1/x).
\]
The following theorem is due to O. Heilmann and E. Lieb. 
\begin{Th}\cite{hei}\label{th:matching}
  All zeros of $M(G,x)$ are real, and the largest zero $t(G)$ is at most $4(\Delta-1)$, where $\Delta$ is the largest real root.
\end{Th}

Now we will present our new proofs for various results of R. Liu and P. Csikv\'ari.

\begin{Cor}\cite{Zhao2005,Csikvari2012a}\label{C:existgamma}
 Let $G$ be a connected graph. Then  $h(G,x)$ has a real zero, and  let $\gamma(G)$ be the largest  among them. Then $\gamma(G)$ is a simple zero of $h(G,x)$, and if $\xi\neq\gamma(G)$ is a zero of $h(G,x)$, then $\gamma(G)>|\xi|$.
\end{Cor}
\begin{proof}
 Choose an ordering of the vertices of $G$, and construct $\widehat{G}$. From the construction of $\widehat{G}$ it is clear that $\widehat{G}$ is a connected graph, so $\beta(\widehat{G})>0$ is a simple simple of $I(\widehat{G},x)=h^*(G,x)$, thus, $\beta(\widehat{G})=\gamma(G)^{-1}$. The rest is the consequence of Theorem~\ref{th:existbeta}.
\end{proof}

%
%


\begin{Cor}\cite{Csikvari2012a}
 Let $G$ be a connected graph.
 Then $\gamma(G)\le t(G)$. In particular  $\gamma(G)\le 4(\Delta-1)$.
\end{Cor}
\begin{proof}
 The first inequality follows from the fact that $\widehat{G}\subseteq L(G)$. Indeed this implies that $$\gamma(G)^{-1}=\beta(\widehat{G})\ge \beta(L(G))=t(G)^{-1},$$ where the inequality follows from Theorem~\ref{th:betamonoton}, and the equalities follow from the identities $h^*(G,x)=x^nI(\widehat{G},1/x)$ and $M(G,x)=x^nI(L(G),1/x)$. The second inequality follows from the first inequality and Theorem~\ref{th:matching}.
\end{proof}

\begin{Cor}\cite{Csikvari2012a}
 Let $H$ be a proper subgraph of $G$. Then $\gamma(H)< \gamma(G)$.
\end{Cor}
\begin{proof}
 Suppose that $H$ can be obtained from $G$ by deleting the edges $\{e_1,\dots,e_k\}$ and then deleting the isolated vertices $\{v_1,\dots,v_l\}$. Let $G_i=G-\{e_1,\dots,e_i\}$ for $1\le i\le k$, and $G_0=G$. Since $G_i$ and $G_{i+1}$ differ only in one edge, then we can choose an ordering of the vertices of $G_i$, such that that $e_i$ is the edge between the last two vertices. Then $\widehat{G_{i+1}}=\widehat{G_i-e}$. This implies that   $\gamma(G_{i+1})<\gamma(G_{i})$ and so
 $\gamma(G_k)< \gamma(G_{k-1})<  \dots \gamma(G_0)=\gamma(G)$.
 
 Since $h(H\cup K_1,x)=(-x)h(H,x)$ and $\gamma(H)>0$, we have that $\gamma(H\cup K_1)=\gamma(H)$. So $\gamma(G_k)=\gamma(G_k-v_1)=\dots=\gamma(G_k-\{v_1,\dots,v_l\})=\gamma(H)$.
\end{proof}

\begin{Cor}\cite{Csikvari2012a}
  Let $G$ be a connected graph and $H$ be a subgraph, then in the following series
  \[\frac{h^*(H,x)}{h^*(G,x)}=\sum_{k\ge 0} s_{k}(G,H) x^k\]
  for each $k\ge 0$ the coefficients $s_k(G,H)$ are positive integers.
\end{Cor}
\begin{proof}
 Direct consequence of Theorem~\ref{th:betamoment}.
\end{proof}

%
%

%
%

\bibliography{hivatkozat}
\bibliographystyle{plain}

\end{document}